\documentclass[graybox]{svmult}

\usepackage{type1cm}        
\usepackage{makeidx}         
\usepackage{graphicx}        
\usepackage{multicol}        
\usepackage[bottom]{footmisc}

\usepackage{newtxtext}       %
\usepackage[varvw]{newtxmath}       

\makeindex             

\begin{document}

\title*{Revisit on global existence of solutions for semilinear damped wave equations in
$\mathbb{R}^N$ with noncompactly supported initial data
}
\titlerunning{Semilinear damped wave equation for noncompactly data}
\author{Yuta Wakasugi}
\institute{Laboratory of Mathematics,
Graduate School of Advanced Science and Engineering,
Hiroshima University,
Higashi-Hiroshima, 739-8527, Japan.
\email{wakasugi@hiroshima-u.ac.jp}
}
%
%
\maketitle

\abstract*{
In this note, we study the Cauchy problem of the semilinear damped wave equation
and our aim is the small data global existence for noncompactly supported initial data.
For this problem, Ikehata and Tanizawa \cite{IkTa05} introduced
the energy method with the exponential-type weight function $e^{|x|^2/(1+t)}$,
which is the so-called Ikehata--Todorova--Yordanov type weight.
In this note, we suggest another weight function of the form $(1+|x|^2/(1+t))^{\lambda}$,
which allows us to treat polynomially decaying initial data
and give a simpler proof than the previous studies treating such initial data.
}

\abstract{In this note, we study the Cauchy problem of the semilinear damped wave equation
and our aim is the small data global existence for noncompactly supported initial data.
For this problem, Ikehata and Tanizawa \cite{IkTa05} introduced
the energy method with the exponential-type weight function $e^{|x|^2/(1+t)}$,
which is the so-called Ikehata--Todorova--Yordanov type weight.
In this paper, we suggest another weight function of the form $(1+|x|^2/(1+t))^{\lambda}$,
which allows us to treat polynomially decaying initial data
and give a simpler proof than the previous studies treating such initial data.
}

\section{Introduction}
\label{sec:1}
In this note, we study the Cauchy problem of the semilinear damped wave equation
\begin{align}\label{eq:ndw}
	\left\{ \begin{alignedat}{3}
	&u_{tt} - \Delta u + u_t = |u|^p,&\quad&(t,x) \in (0,\infty) \times \mathbb{R}^N,\\
	&u(0,x) = u_0(x), \ u_t(0,x) = u_1(x),&\quad& x \in \mathbb{R}^N,
\end{alignedat} \right.
\end{align}
where $u = u(t,x)$ is a real-valued unknown,
$x = (x_1,\ldots,x_N) \in \mathbb{R}^N$,
$N \ge 1$,
$u_t = \partial u/\partial t, u_{tt} = \partial^2 u/\partial t^2$,
$\Delta u = \partial^2 u/\partial x_1^2 + \cdots + \partial^2 u/ \partial x_N^2$,
$p>1$,
and
$u_0, u_1$ are given initial data.

The purpose is to show the small data global existence of the solution
with noncompactly supported initial data.
For this problem, Ikehata--Tanizawa \cite{IkTa05} introduced
the energy method with the exponential-type weight function $e^{|x|^2/(1+t)}$,
which is the so-called Ikehata--Todorova--Yordanov type weight.
They proved the small data global existence for
$p > p_F(N) := 1+2/N$ and exponentially decaying initial data.
This result removed the compactness assumption on the support of initial data
in the earlier result by Todorova--Yordanov \cite{ToYo01}.

In this note, we suggest another weight function of the form
$(1+ \frac{|x|^2}{1+t})^{\lambda}$ with sufficiently large $\lambda$.
This allows us to weaken the assumption of \cite{IkTa05} to polynomially decaying initial data.
Actually, this type of weight was already used in Ikehata--Nishihara--Zhao \cite{IkNiZh06} to study 
the case of absorbing nonlinearity.
However, it seems to have been overlooked that this technique is also effective in the case of sourcing nonlinearity.

To state our results, we first define the mild solution.
Let
$\mathcal{D}(t)$
be the fundamental solution of the linear damped wave equation, that is,
$\mathcal{D}(t)$ is defined by the Fourier multiplier
\[
	\mathcal{D}(t) = \mathcal{F}^{-1} e^{-\frac{t}{2}} \frac{\sinh (t \sqrt{1/4-|\xi|^2})}{\sqrt{1/4-|\xi|^2}} \mathcal{F},
\]
where
$\mathcal{F}, \mathcal{F}^{-1}$
denote the Fourier and inverse Fourier transforms, respectively.

\begin{definition}\label{def:sol}
Let $T \in (0, \infty]$ and
$I = [0,T]$ if $T < \infty$ and $I = [0,\infty)$ if $T = \infty$.
Let $(u_0, u_1) \in H^1(\mathbb{R}^N) \times L^2(\mathbb{R}^N)$.
We say that
$u \in C(I;H^1(\mathbb{R}^N)) \cap C^1(I;L^2(\mathbb{R}^N))$
is a mild solution of the Cauchy problem \eqref{eq:ndw} on $I$ if
$u$ satisfies
\[
	u(t) = \mathcal{D}(t) (u_0+u_1) + \frac{\partial}{\partial t} \mathcal{D}(t) u_0
	+ \int_0^t \mathcal{D}(t-s) |u(s)|^p \,ds
\]
in $C(I;H^1(\mathbb{R}^N)) \cap C^1(I;L^2(\mathbb{R}^N))$.
In particular, if $T < \infty$, then $u$ is called a local (in time) mild solution;
if $T = \infty$, then $u$ is called a global (in time) mild solution.
\end{definition}

The main result of this note is the following.

\begin{theorem}\label{thm:1}
Let $N \ge 1$ and
let $p > 1$ satisfy
\[
	p_F(N) < p \le \frac{N}{[N-2]_+},
\]
where $p_F(N) = 1+2/N$.
Then, there exist constants $\lambda=\lambda(N,p) > 0$ and $\varepsilon = \varepsilon (N, p, \lambda) > 0$
 such that
if the initial data
 $(u_0, u_1) \in H^1(\mathbb{R}^N) \times L^2(\mathbb{R}^N)$
 satisfy
 \begin{equation}\label{eq:ini}
	\| (1+|x|)^{\lambda} u_1 \|_{L^2}
	+ \| (1+|x|)^{\lambda} \nabla_x u_0 \|_{L^2}
	+ \| (1+|x|)^{\lambda} u_0 \|_{L^2} < \varepsilon,
 \end{equation}
 then the Cauchy problem \eqref{eq:ndw} admits the unique global mild solution.
\end{theorem}
\begin{remark}
The notation
$p \le \frac{N}{[N-2]_+}$
means that
$p < \infty$ for $N \le 2$ and $p \le \frac{N}{N-2}$ for $N\ge 3$.
Also, we may take
$\lambda$ in the theorem so that
$\lambda > \max\{ 1, \frac{N}{2} \}$
and
\[
	\lambda > \max\left\{ \dfrac{2N-(N-2)p}{2p}, \dfrac{2N-8/N-(N-2)p}{4(p-p_F(N))}, \dfrac{q-1}{q}\dfrac{N+2-(N-2)p}{2(p-p_F(N))} \right\}
\]
with $q = \max\{ 2, \frac{N}{2}(p-1) \}$
hold.
Note that
$\lambda > \frac{N}{2}$ ensures
$u_0, \nabla_x u_0, u_1 \in L^1(\mathbb{R}^N)$.
\end{remark}

\begin{remark}
The conclusion of Theorem \ref{thm:1} itself is weaker than those of
Hayashi--Kaikina--Naumkin \cite{HaKaNa04DIE} and Sobajima \cite{So19DIE}
in the sense that our order of the weight $\lambda$ (see Remark 1 above) has to be larger than those of
\cite{HaKaNa04DIE, So19DIE}.
However, the method of the proof of Theorem \ref{thm:1}
is more elementary and it will also be applicable to
the case of time-dependent damping such as
\[
	u_{tt} - \Delta u + \frac{\mu}{1+t} u_t = |u|^p.
\]
It will slightly improve the assumption on the initial data in the result of D'Abbicco \cite[Theorem 3]{Da15}.
Moreover, combining with the linear estimates for damped wave equations in general measure spaces proved in
\cite{IkTaWa}, we will be able to obtain
global existence results for semilinear damped wave equations 
in more general domains or with some potentials, etc.
\end{remark}

Finally, we introduce some notations used in this note.
The set of nonnegative integers is expressed by $\mathbb{Z}_{\ge 0}$.
The symbol $f \lesssim g$ stands for $f \le C g$ with some constant $C > 0$. 
The gradient with respect to the spatial variables is written as $\nabla_x$,
that is,
$\nabla_x u = (u_{x_1}, \ldots, u_{x_N})$.
$L^q(\mathbb{R}^N)$ and $H^s(\mathbb{R}^N)$ denote the usual Lebesgue and Sobolev spaces, respectively.

\section{Energy estimate with a polynomial-type weight}
The main idea of this note is to use the following weighted energy functional.
\label{sec:2}
\begin{definition}\label{def:weight}
For $\lambda> 0$ and $A \ge \lambda/2$,
we define the weight functions by
\begin{equation}
	\psi(t,x) = A + \frac{|x|^2}{1+t},\quad \Psi(t,x) = \psi(t,x)^{\lambda}
\end{equation}
and the weighted energy of a function $u$ by
\[
	E_{\Psi}(t) = \int_{\mathbb{R}^N} ( |u_t(t,x)|^2 + |\nabla_x u(t,x)|^2 ) \Psi(t,x) \,dx.
\]
\end{definition}
\begin{lemma}\label{lem:we}
The weight function $\Psi$ satisfies
$\Psi_t \le 0$
and
\[
	\frac{|\nabla_x \Psi|^2}{- \Psi_t} \le 2 \Psi.
\]
\end{lemma}
\begin{proof}
We calculate
\[
	\Psi_t = \lambda \psi^{\lambda-1} \psi_t = - \lambda \frac{|x|^2}{4(1+t)^2} \psi^{\lambda-1} < 0.
\]
Moreover, since $A \ge \lambda/2$, then we have
\[
	\frac{|\nabla_x \Psi|^2}{-\Psi_t}
	= \frac{ \lambda^2 \frac{|x|^2}{4(1+t)^2} \psi^{2(\lambda-1)}}{\lambda \frac{|x|^2}{4(1+t)^2} \psi^{\lambda-1}}
	= \lambda \psi^{\lambda-1} 
	= \frac{\lambda}{\psi} \Psi \le \frac{\lambda}{A} \Psi \le 2 \Psi.
\]
This completes the proof.
\end{proof}

\begin{lemma}\label{lem:en}
Let $u$ be a mild solution to \eqref{eq:ndw} on $[0,T]$.
 Then, we have
 \begin{align*}
 	E_{\Psi}(t) &\le E_{\Psi}(0) - \frac{2}{p+1} \int_{\mathbb{R}^N} |u_0(x)|^p u_0(x) \Psi(0,x) \,dx \\
	&\quad
	+ \frac{2}{p+1} \int_{\mathbb{R}^N} |u(t,x)|^p u(t,x) \Psi(t,x) \,dx \\
	&\quad
	- \frac{2}{p+1} \int_0^t \int_{\mathbb{R}^N} |u(s,x)|^p u(s,x) \Psi_t(s,x) \,dxds,
 \end{align*}
 provided that the right-hand side is finite.
\end{lemma}
\begin{proof}
Differentiating $E_{\Psi}(t)$, applying integration by parts and using the equation \eqref{eq:ndw},
we have
\footnote{By an approximation argument, we can justify the fornal calclulation here, see e.g., \cite{Wa23}}
\begin{align*}
	\frac{d}{dt} E_{\Psi}(t)
	&=
	2 \int_{\mathbb{R}^N} \left( u_t u_{tt} + \nabla_x u \cdot \nabla_x u_t \right) \Psi \,dx
	+ \int_{\mathbb{R}^N} \left( |u_t|^2 + |\nabla_x u|^2 \right) \Psi_t \,dx \\
	&=
	-2 \int_{\mathbb{R}^N} u_t^2 \Psi \,dx + 2 \int_{\mathbb{R}^N} |u|^p u_t \Psi \,dx \\
	&\quad
	- 2 \int_{\mathbb{R}^N} u_t (\nabla_x u \cdot \nabla_x \Psi) \,dx
	+ \int_{\mathbb{R}^N} \left( |u_t|^2 + |\nabla_x u|^2 \right) \Psi_t \,dx.
\end{align*}
Here, by Lemma \ref{lem:we}, we note that $\Psi_t \le 0$ and
\begin{align*}
	-2 u_t (\nabla_x u \cdot \nabla_x \Psi) + |\nabla_x u|^2 \Psi_t
	=
	\frac{1}{\Psi_t} \left| \Psi_t \nabla_x u - u_t \nabla_x \Psi \right|^2 - \frac{|\nabla_x \Psi|^2}{\Psi_t} u_t^2
	\le 2 \Psi u_t^2
\end{align*}
hold.
Using them to the previous identity, we deduce
\begin{align*}
	\frac{d}{dt} E_{\Psi}(t)
	&\le
	2 \int_{\mathbb{R}^N} |u|^p u_t \Psi \,dx \\
	&=
	\frac{2}{p+1} \frac{d}{dt} \int_{\mathbb{R}^N} |u|^p u \Psi \,dx
	- \frac{2}{p+1} \int_{\mathbb{R}^N} |u|^p u \Psi_t \,dx.
\end{align*}
Integrating it over $[0,t]$ gives the conclusion.
\end{proof}

\section{Proof of Theorem \ref{thm:1}}
\subsection{Preliminaries}
The following local existence result can be proved in completely the same way as
Ikehata--Tanizawa \cite[Proposition 2.1]{IkTa05}, and we omit the proof.
\begin{lemma}\label{lem:le}
Let $N \ge 1$, $1<p \le \frac{N}{[N-2]_+}$.
Then for each $(u_0, u_1) \in H^1(\mathbb{R}^N) \cap L^2(\mathbb{R}^N)$ satisfying
\[
	\| (1+|x|)^{\lambda} u_1 \|_{L^2}
	+ \| (1+|x|)^{\lambda} \nabla_x u_0 \|_{L^2}
	+ \| (1+|x|)^{\lambda} u_0 \|_{L^2} < \infty,
\]
there exists $T > 0$ such that the Cauchy problem \eqref{eq:ndw} admits
a unique mild solution on $[0,T]$ satisfying
\[
	\| \Psi^{1/2} u_t(t) \|_{L^2}
	+ \| \Psi^{1/2} \nabla_x u(t) \|_{L^2}
	+ \| \Psi^{1/2} u(t) \|_{L^2} < \infty
\]
for $t \in [0,T]$.
Moreover, if the maximal existence time $T_m$, the supremum of $T$ such that
there is a unique mild solution on $[0,T]$, is finite, then the solution must satisfy
\[
	\limsup_{t \uparrow T_m} \left( \| \Psi^{1/2} u_t(t) \|_{L^2}
	+ \| \Psi^{1/2} \nabla_x u(t) \|_{L^2}
	+ \| \Psi^{1/2} u(t) \|_{L^2} \right)
	= + \infty.
\]
\end{lemma}

We define the following norm to construct the global solution.
\begin{definition}\label{def:X}
We define the norm $\| \cdot \|_{X(T)}$ by
\begin{align*}
	\| u \|_{X(T)}
	&:= \sup_{0<t<T} \left[
	E_{\Psi}(t)^{\frac{1}{2}}
	+ (1+t)^{\frac{N}{4}+1} \| u_t (t) \|_{L^2} \right. \\
	&\qquad \qquad \left.
	+ (1+t)^{\frac{N}{4}+\frac{1}{2}} \| \nabla_x u(t) \|_{L^2} 
	+ (1+t)^{\frac{N}{4}} \| u(t) \|_{L^2}
	\right].
\end{align*}
\end{definition}
By the local existence result in Lemma \ref{lem:le} and noting
\[
	\| \Psi^{1/2} u(t) \|_{L^2} = \left\| \Psi^{1/2} \left( u_0 + \int_0^t u_t (s) \,ds \right) \right\|_{L^2}
	\lesssim \| (1+|x|)^{\lambda} u_0 \|_{L^2} + T \|u \|_{X(T)},
\]
it suffices to show the a priori estimate
$\| u \|_{X(T)} \le C$ with some constant $C$ independent of $T$
under the smallness assumption \eqref{eq:ini}.

In order to estimate the $L^2$-norms in $\| u \|_{X(T)}$, we apply the following
Matsumura estimate.
\begin{lemma}[{Matsumura \cite{Ma76}}]\label{lem:Ma}
For $k \in \mathbb{Z}_{\ge 0}, \alpha \in \mathbb{Z}_{\ge 0}^N$, $q \in [1,2]$
and $g \in L^q(\mathbb{R}^N) \cap H^{k+|\alpha|-1}(\mathbb{R}^N)$, we have
\[
	\| \partial_t^k \partial_x^{\alpha} \mathcal{D}(t) g \|_{L^2}
	\lesssim (1+t)^{-\frac{N}{2}\left(\frac{1}{q}-\frac{1}{2}\right)-k-\frac{|\alpha|}{2}}
	\left( \| g \|_{L^q} + e^{-\frac{t}{4}} \| g \|_{H^{k+|\alpha|-1}} \right)
\]
\end{lemma}
Finally, to estimate the nonlinear term with polynomial weights,
we employ the Caffarelli--Kohn--Nirenberg inequality.
\begin{lemma}[{Caffarelli--Kohn--Nirenberg \cite{CaKoNi84}}] \label{lem:CaKoNi}
Let
$p, q, r, \alpha, \beta, \sigma, a \in \mathbb{R}$ satisfy
$p,q \ge 1$, $r > 0$, $0 \le a \le 1$,
\[
	\frac{1}{p} + \frac{\alpha}{N}, \frac{1}{q}+\frac{\beta}{N}, \frac{1}{r}+\frac{\gamma}{N} > 0,
\]
where
$\gamma = a \sigma + (1-a) \beta$.
Then, there exists a positive constant $C$ such that the following inequality holds for all
$u \in C_0^{\infty}(\mathbb{R}^N)$
\[
	\| |x|^{\gamma} u \|_{L^r} \le C \| |x|^{\alpha} \nabla_x u \|_{L^p}^a \| |x|^{\beta} u \|_{L^q}^{1-a}
\]
if and only if the following relations hold:
\[
	\frac{1}{r} + \frac{\gamma}{N}
	= a \left( \frac{1}{p} + \frac{\alpha-1}{N} \right) + (1-a) \left( \frac{1}{q} + \frac{\beta}{N} \right)
\]
(this is the dimensional balance),
\[
	0 \le \alpha -\sigma \quad \text{if} \quad a > 0,
\]
and
\[
	\alpha - \sigma \le 1 \quad \text{if} \quad a > 0 \quad
	\text{and} \quad \frac{1}{p} + \frac{\alpha-1}{N} = \frac{1}{r} + \frac{\gamma}{N}.
\]
\end{lemma}
In particular, the case
$\alpha = \beta = \gamma = 0$
is called the Gagliardo--Nirenberg inequality.

\subsection{Proof of Theorem \ref{thm:1}}
We first prove that the semilinear terms in Lemma \ref{lem:en} is bounded by $\| u \|_{X(T)}$.
\begin{lemma}\label{lem:nl1}
Let $u$ be a mild solution to \eqref{eq:ndw} on $[0,T]$.
Then, we have
\[
	\int_{\mathbb{R}^N} |u(t,x)|^{p+1} \Psi(t,x) \,dx
	+ \int_0^t \int_{\mathbb{R}^N} |u(s,x)|^{p+1} |\Psi_t(s,x)| \,dxds
	\le C \| u \|_{X(T)}^{p+1},
\]
with some constant $C>0$, provided that the right-hand side is finite.
\end{lemma}
\begin{proof}
We consider the first term of the left-hand side.
Noting
\[
	\int_{\mathbb{R}^N} |u|^{p+1} \Psi \,dx
	\lesssim \| u \|_{L^{p+1}}^{p+1} + (1+t)^{-\lambda} \| |x|^{\frac{2\lambda}{p+1}} u \|_{L^{p+1}}^{p+1}
\]
and applying Lemma \ref{lem:CaKoNi}, we have
\begin{align*}
	\| u \|_{L^{p+1}}
	&\lesssim \| \nabla_x u \|_{L^2}^{\theta} \| u \|_{L^2}^{1-\theta},
	\quad
	\theta = \frac{N(p-1)}{2(p+1)},\\
	\| |x|^{\frac{2\lambda}{p+1}} u \|_{L^{p+1}}
	&\lesssim \| |x|^{\lambda} \nabla_x u \|_{L^2}^{\Theta} \| u \|_{L^q}^{1-\Theta},
	\quad
	\Theta = \frac{\frac{1}{p+1}+\frac{2\lambda}{N(p+1)}-\frac{1}{q}}{\frac{1}{2}-\frac{1}{q}+\frac{\lambda-1}{N}},
\end{align*}
where 
$q = \max\{ 2, \frac{N}{2}(p-1) \}$.
Here, we remark that the condition $0\le \alpha - \sigma$ if $a > 0$ in Lemma \ref{lem:CaKoNi} requires
$p \le 1+\frac{2q}{N}$.
Moreover, $H^1(\mathbb{R}^N) \subset L^q(\mathbb{R}^N)$ holds, since
$p \le 1+\frac{4}{[N-2]_+}$.
They are the reasons why we take $q = \max\{ 2, \frac{N}{2}(p-1) \}$.
The Gagliardo--Nirenberg inequality further implies
\[
	\| u \|_{L^q} \lesssim \| \nabla_x u \|_{L^2}^{\mu} \| u \|_{L^2}^{1-\mu},
	\quad
	\mu = N \left( \frac{1}{2} - \frac{1}{q} \right).
\]
Therefore, we conclude
\begin{align*}
	\int_{\mathbb{R}^N} |u|^{p+1} \Psi \,dx
	&\lesssim
	\left( \| \nabla_x u \|_{L^2}^{\theta} \| u \|_{L^2}^{1-\theta} \right)^{p+1} \\
	&\quad
	+ (1+t)^{\lambda\left( \frac{p+1}{2}\Theta -1 \right)}
	\left( \left\| \frac{|x|^{\lambda}}{(1+t)^{\lambda/2}} \nabla_x u \right\|_{L^2}^{\Theta}
	(\| \nabla_x u \|_{L^2}^{\mu} \| u \|_{L^2}^{1-\mu} )^{1-\Theta} \right)^{p+1} \\
	&\lesssim
	\left[
	(1+t)^{-\left(\frac{N}{4}+\frac{\theta}{2}\right)(p+1)}
	+ (1+t)^{\lambda\left( \frac{p+1}{2}\Theta -1 \right) - \left( \frac{N}{4}+\frac{\mu}{2}\right) (1-\Theta)(p+1)}
	\right]
	\| u \|_{X(T)}^{p+1} \\
	&\lesssim
	\| u \|_{X(T)}^{p+1},
\end{align*}
since a straightforward computation shows
$\lambda\left( \frac{p+1}{2}\Theta -1 \right) - \left( \frac{N}{4}+\frac{\mu}{2}\right) (1-\Theta)(p+1) < 0$
if
$\lambda > \dfrac{q-1}{q} \dfrac{N+2-(N-2)p}{2(p-p_F(N)}$.
Moreover, noticing
$|\Psi_t| \lesssim (1+t)^{-1} \Psi$,
we can estimate the second term of the left-hand side of Lemma \ref{lem:nl1} in the same way. 
This completes the proof.
\end{proof}

Next, we estimate the term $(1+t)^{N/4} \| u \|_{L^2}$ in the norm $\| u \|_{X(T)}$. 
\begin{lemma}\label{lem:L2}
Let $u$ be a mild solution to \eqref{eq:ndw} on $[0,T]$.
Then, we have
\[
	(1+t)^{\frac{N}{4}} \left\| \int_0^t \mathcal{D}(t-s) |u(s)|^p \,ds \right\|_{L^2}
	\lesssim \| u \|_{X(T)}^{p}.
\]
\end{lemma}
\begin{proof}
By Lemma \ref{lem:Ma}, we estimate
\begin{align*}
	&(1+t)^{\frac{N}{4}} \left\| \int_0^t \mathcal{D}(t-s) |u(s)|^p \,ds \right\|_{L^2} \\
	&\lesssim
	(1+t)^{\frac{N}{4}} \int_0^{\frac{t}{2}} (1+t-s)^{-\frac{N}{4}}
			\left( \| |u(s)|^p \|_{L^1} + e^{-\frac{t-s}{4}} \| |u(s)|^p \|_{L^2} \right) \,ds \\
	&\quad
	+ (1+t)^{\frac{N}{4}} \int_{\frac{t}{2}}^t \| |u(s)|^p \|_{L^2} \,ds.
\end{align*}
In order to estimate the term $\| |u(s)|^p \|_{L^1} = \| u(s) \|_{L^p}^p$,
we divide the case into $p < 2$ and $p \ge 2$.
When $p < 2$, applying Lemma \ref{lem:CaKoNi}, we have
\[
	\| u \|_{L^p} \lesssim \| |x|^{\lambda} \nabla_x u \|_{L^2}^{\theta} \| u \|_{L^2}^{1-\theta},
	\quad
	\theta = \frac{N(2-p)}{2p(\lambda -1)}.
\]
We remark that $\theta <1$ holds, since $\lambda > \dfrac{2N-(N-2)p}{2p}$.
Thus, we have
\begin{align*}
	\| u \|_{L^p}^p
	&\lesssim
	(1+t)^{\frac{\lambda p}{2} \theta - \frac{N}{4}p(1-\theta)}
	\left( \left\| \left( \frac{|x|}{\sqrt{1+t}} \right)^{\lambda} \nabla_x u \right\|_{L^2}^{\theta}
		\left( (1+t)^{\frac{N}{4}} \| u \|_{L^2} \right)^{1-\theta} \right)^p \\
	&\lesssim 
	(1+t)^{\frac{\lambda p}{2} \theta - \frac{N}{4}p(1-\theta)} \| u \|_{X(T)}^p
\end{align*}
A direct computation shows that
$\frac{\lambda p}{2} \theta - \frac{N}{4}p(1-\theta) < -1$,
since $\lambda > \frac{2N-8/N-(N-2)p}{4(p-p_F(N))}$.
When $p \ge 2$, the Gagliardo--Nirenberg inequality implies
\[
	\| u \|_{L^p} \lesssim \| \nabla_x u \|_{L^2}^{\theta} \| u \|_{L^2}^{1-\theta},\quad
	\theta = \frac{N(p-2)}{2p}.
\]
Then, we easily obtain
\[
	\| u \|_{L^p}^p \lesssim
	(1+t)^{-\left(\frac{N}{4}+\frac{1}{2}\right)p \theta - \frac{N}{4}p(1-\theta)} \| u \|_{X(T)}^p
\]
with $-\left(\frac{N}{4}+\frac{1}{2}\right)p \theta - \frac{N}{4}p(1-\theta) < -1$,
since $p > p_F(N)$.
Therefore, in both cases, we conclude
\[
	\int_0^{\frac{t}{2}} \| |u(s)|^p \|_{L^1} \,ds \lesssim \| u \|_{X(T)}^p.
\]
Next, we estimate the term $\| |u(s)|^p \|_{L^2} = \| u(s) \|_{L^{2p}}^p$.
Applying the Gagliardo--Nirenberg inequality, we have
\[
	\| u \|_{L^{2p}} \lesssim \| \nabla_x u \|_{L^2}^{\theta} \| u \|_{L^2}^{1-\theta},
	\quad
	\theta = \frac{N(p-1)}{2p}.
\]
We remark that $\theta \le 1$ holds, since $p \le \frac{N}{[N-2]_+}$.
Using the above estimate and repeat a similar computation as above, we have
\[
	(1+t)^{\frac{N}{4}} \int_0^{\frac{t}{2}}  e^{-\frac{t-s}{4}} \| |u(s)|^p \|_{L^2}  \,ds
	+ (1+t)^{\frac{N}{4}} \int_{\frac{t}{2}}^t \| |u(s)|^p \|_{L^2} \,ds
	\lesssim \| u \|_{X(T)}^p,
\]
which completes the proof.
\end{proof}

In the same way, we can estimate the other terms in the norm $\| u \|_{X(T)}$.

\begin{lemma}\label{lem:dL2}
Let $u$ be a mild solution to \eqref{eq:ndw} on $[0,T]$.
For nonnegative integers $k, j$ satisfying $k+j \le 1$,
we have
\[
	(1+t)^{\frac{N}{4}+k+\frac{j}{2}} \left\| \int_0^t \partial_t^k \nabla_x^j \mathcal{D}(t-s) |u(s)|^p \,ds \right\|_{L^2}
	\lesssim \| u \|_{X(T)}^{p}.
\]
\end{lemma}
\begin{proof}
By Lemma \ref{lem:Ma}, we have
\begin{align*}
	&(1+t)^{\frac{N}{4}+k+\frac{j}{2}} \left\| \int_0^t \partial_t^k \nabla_x^j \mathcal{D}(t-s) |u(s)|^p \,ds \right\|_{L^2} \\
	&\lesssim
	(1+t)^{\frac{N}{4}+k+\frac{j}{2}}
	\int_0^{\frac{t}{2}} (1+t-s)^{-\frac{N}{4}-k-\frac{j}{2}}
		\left( \| |u(s)|^p \|_{L^1} + e^{-\frac{t-s}{4}} \| |u(x)|^p \|_{L^2} \right)\,ds \\
	&\quad
	+ (1+t)^{\frac{N}{4}+k+\frac{j}{2}} \int_{\frac{t}{2}}^t (1+t-s)^{-k-\frac{j}{2}} \| |u(s)|^p \|_{L^2} \,ds.
\end{align*}
The rest part is completely the same as Lemma \ref{lem:L2} and we omit the detail.
\end{proof}

Now let us prove Theorem \ref{thm:1}.
By Lemmas \ref{lem:en} and \ref{lem:nl1}, we have
\begin{align*}
	E_{\Psi}(t)^{1/2} \lesssim E_{\Psi}(0)^{1/2} + \| (1+|x|)^{\frac{\lambda}{p+1}} u_0 \|_{L^{p+1}}^{\frac{p+1}{2}}
						+ \| u \|_{X(T)}^{\frac{p+1}{2}}.
\end{align*}
Applying Lemma \ref{lem:CaKoNi} to the second term, the right-hand side can be further bounded by
$C (\varepsilon + \varepsilon^{\frac{p+1}{2}}) + C \| u \|_{X(T)}^{\frac{p+1}{2}}$.
Next, by Lemmas \ref{lem:Ma}, we can easily prove that
\begin{align*}
	&(1+t)^{\frac{N}{4}+1} \| u^L_t (t) \|_{L^2}
	+ (1+t)^{\frac{N}{4}+\frac{1}{2}} \| \nabla_x u^L(t) \|_{L^2} 
	+ (1+t)^{\frac{N}{4}} \| u^L(t) \|_{L^2} \\
	&\lesssim
	\| u_0 \|_{L^1} + \| u_0 \|_{H^1} + \| u_1 \|_{L^1} + \| u_1 \|_{L^2}
	\lesssim \varepsilon
\end{align*}
for the linear part of the solution
$u^L := \mathcal{D}(t) (u_0+u_1) + \frac{\partial}{\partial t} \mathcal{D}(t) u_0$.
Moreover, by Lemmas \ref{lem:L2} and \ref{lem:dL2}, we have
\begin{align*}
	&(1+t)^{\frac{N}{4}+1} \left\| \int_0^t \partial_t \mathcal{D}(t-s) |u(s)|^p \,ds \right\|_{L^2}
	+(1+t)^{\frac{N}{4}+\frac{1}{2}} \left\| \int_0^t \nabla_x \mathcal{D}(t-s) |u(s)|^p \,ds \right\|_{L^2} \\
	&+ (1+t)^{\frac{N}{4}} \left\| \int_0^t \mathcal{D}(t-s) |u(s)|^p \,ds \right\|_{L^2} \\
	&\lesssim 
	\| u \|_{X(T)}^{p}.
\end{align*}
Putting all the above estimates together, we conclude
\[
	\| u \|_{X(T)} \lesssim \varepsilon + \varepsilon^{\frac{p+1}{2}} + \| u \|_{X(T)}^{\frac{p+1}{2}} + \| u \|_{X(T)}^{p},
\]
where the implicit constant is independent of $T$.
From this and a standard argument give the a priori estimate
$\| u\|_{X(T)} \le C$ with some constant $C>0$ independent of $T$,
provided that $\varepsilon$ is sufficiently small,
and the proof is complete.

\section*{Acknowledgements}
This work was supported by JSPS KAKENHI Grant Number JP20K14346.
The author is deeply grateful to Professors Masahiro Ikeda, Koichi Taniguchi and Motohiro Sobajima
for the fruitful discussion.


\begin{thebibliography}{9}
\bibitem{CaKoNi84}\textsc{L. Caffarelli, R. Kohn, L. Nirenberg},
\textit{First order interpolation inequalities with weights},
Compositio Math. \textbf{53} (1984), 259--275. 

\bibitem{Da15}\textsc{M. D'Abbicco},
\textit{The threshold of effective damping for semilinear wave equations},
Math. Methods Appl. Sci. \textbf{38} (2015), 1032--1045.

\bibitem{HaKaNa04DIE}\textsc{N. Hayashi, E. I. Kaikina, P. I. Naumkin},
\textit{Damped wave equation with super critical nonlinearities},
Differential Integral Equations \textbf{17} (2004), 637--652.

\bibitem{IkTaWa}\textsc{M. Ikeda, K. Taniguchi, Y. Wakasugi},
\textit{Global existence and asymptotic behavior for semilinear damped wave equations on measure spaces},
arXiv:2106.10322v3.

\bibitem{IkTa05}\textsc{R. Ikehata, K. Tanizawa},
\textit{Global existence of solutions for semilinear damped wave equations in $\mathbf{R}^N$
with noncompactly supported initial data},
Nonlinear Anal. \textbf{61} (2005), 1189--1208.

\bibitem{IkNiZh06}\textsc{R. Ikehata, K. Nishihara, H. Zhao},
\textit{Global asymptotics of solutions to
the Cauchy problem for the damped wave equation with absorption},
J. Differential Equations \textbf{226} (2006), 1--29.

\bibitem{Ma76}\textsc{A. Matsumura},
\textit{On the asymptotic behavior of solutions of semi-linear wave equations},
Publ. RIMS, Kyoto University \textbf{12} (1976) 169--189.

\bibitem{So19DIE}\textsc{M. Sobajima},
\textit{Global existence of solutions to semilinear damped wave equation with slowly decaying
initial data in exterior domain},
Differential Integral Equations \textbf{32} (2019), 615--638.

\bibitem{ToYo01}\textsc{G. Todorova, B. Yordanov},
\textit{Critical exponent for a nonlinear wave equation with damping},
J. Differential Equations \textbf{174} (2001), 464--489.

\bibitem{Wa23}\textsc{Y. Wakasugi},
\textit{Decay property of solutions to the wave equation with space-dependent damping, absorbing nonlinearity, and polynomially decaying data},
Math. Methods Appl. Sci. \textbf{46} (2023), 7067--7107.


\end{thebibliography}
\end{document}